\documentclass[11pt,reqno]{amsart}
\usepackage{amscd,amsmath,amsopn,amssymb,amsthm,multicol}
\usepackage{tikz,subdepth,anysize,verbatim,ifthen,xargs,colortbl,float}
\usepackage{longtable,mathtools,hyperref}
\usepackage[english]{babel}
\everymath=\expandafter{\the\everymath\displaystyle}

\theoremstyle{plain}
\newtheorem{theorem}{Theorem}
\newtheorem{prop}{Proposition}

\newtheorem{remark}{Remark}
\theoremstyle{definition}
\newtheorem{definition}{Definition}

\newcommand\com[1]{}
\newcommand\C{{\mathbb C}}

\newcommand\E{\mathcal{E}}
\newcommand\g{{\mathfrak g}}
\newcommand\h{\mathfrak{h}}

\newcommand\op[1]{\mathop{\rm #1}\nolimits}
\newcommand\p{\partial}

\newcommand\R{{\mathbb R}}
\newcommand\Z{{\mathbb Z}}

\begin{document}

\title[On radicals in differential invariants]{On radicals in differential invariants: \\ global theory vs moving frames}
\author{Boris Kruglikov}
\address{Department of Mathematics and Statistics, UiT the Arctic University of Norway, Troms\o\ 9037, Norway.
\ E-mail: {\tt boris.kruglikov@uit.no}. }

 \begin{abstract}
The classical technique of moving frames allows to effectively compute differential invariants of group actions. 
A standard observation is that the expressions of differential invariants contain radicals and more complicated 
algebraic functions, however the global Lie-Tresse theorem claims that for algebraic pseudogroup actions 
the algebra of differential invariants is generated by rational functions in jet variables. 
Here, we explain this apparent contradiction and illustrate it in a few well-elaborated examples.
Constructively, our result gives a method to algorithmically compute global differential invariants.
 \end{abstract}

\maketitle

\section{An overview of the problem and the result}\label{S1}

Differential invariants of continuous group actions were introduced and studied by Sophus Lie \cite{Lie2}. 
He and his student Arthur Tresse proved the first (local) version of finite-generation of the algebra of invariants \cite{Lie1,Tr1}.
This Lie-Tresse theorem was later generalized in several directions \cite{C3,Ku,Ov,OP,KL2,OPV}.
The methods of computing differential invariants varied from parameter exclusions to infinitesimal actions.

One of the most widespread tools, the method of moving frames, was elaborated by Élie Cartan \cite{C1}, see also \cite{Gri,Gre,FO1}.
It is based on normalizations of group parameters via cross-sections to the orbits in the phase space. 
Numerous applications of this method produce differential invariants involving radicals, as is clearly seen from
computations of fundamental invariants and basic derivations in various geometries \cite{T,Lie1,Lie2,CSW,MB,KL1,MMR}.
For instance, such are the curvature and the differential parameter of smooth curves $y=y(x)$ in the Euclidean plane $\R^2(x,y)$
 $$
\kappa=\frac{y_{xx}}{(1+y_x^2)^{3/2}},\qquad \frac{d}{ds}=\frac1{\sqrt{1+y_x^2}}\frac{d}{dx},
 $$
which were already discussed in \cite{KL2} to be non-invariant with respect to the moton group $SO(2)\ltimes\R^2$.

This occurrence of radicals is at odds with the global Lie-Tresse theorem obtained in \cite{KL2}, which
states that the algebra of differential invariants is generated by rational functions in jet variables for algebraic pseudogroup 
actions transitive on the base manifold. In the above example the generators of the algebra of global invariants are
 $$
K=\kappa^2,\qquad \nabla=\kappa\cdot \frac{d}{ds},
 $$
that is, every global differential invariant can be expressed through a rational combination of $K,\nabla K,\nabla^2K,\dots$
In other words, these rational differential invariants are coordinates on the quotient equation.

An explanation of this apparent contradiction lies in the fact that the differential invariants produced by the method
of moving frames are not global in jets. 
(In the particular case of curves in the Euclidean plane, global invariants form a subfield of degree 2 in
the field of local invariants generated by $\kappa,\kappa_s,\kappa_{ss},\dots$)
In fact, the cross-sections introduced in normalization conditions are often only local transversals, 
and what is called free actions in the literature, as a base for applications of the moving frames, 
turns out to be almost locally free actions. This explains the locality and the multitude of roots in the classical formulae.

In fact, already Élie Cartan noted in his seminal five-variables paper \cite{C0}, after introducing the invariant $I$
in the particular case of 4-multiple root for his fundamental quartic of (2,3,5)-distributions:
 \begin{quote}
 {\it\small ``Il est bon de remarquer qu'on est arrivé à l'invariant $I$ par des extractions de racine $4^\text{e}$
 et de racine carrée, de sorte qu'en realité ce n'est pas $I$ lui-meme qui est un invariant, 
 mais une de ses puissances.''}
 \end{quote}
Then he mentioned the importance of using {\it ``invariants rationels''}. Yet, up to this date, a majority of 
investigations on differential invariants is based on radical expressions.

In this note we first review the method of moving frames, as was formulated by Élie Cartan and then refined and 
expanded by Shiing-Shen Chern, Phillip Griffiths, Peter Olver, Mark Fels and others, see \cite{Gri,Gar,FO2}.
Afterwards we overview examples and explain a relation of the algebras of global and local differential invariants;
the later is known in the literature as the the algebra of differential invariants. 
The approach to global differential invariants is based on joint work with Valentin Lychagin \cite{KL2}.

We compare the two theories on their common intersection ground. Thus we assume that the invariants via moving frames
are obtained through invariantization of the field of rational functions on a local semi-algebraic section $\Sigma'$.
In purely algebraic context a relation between moving frames and rational invariants was clarified in \cite{HK1,HK2} with 
an emphasis on constructive methods.
In more general differential context we get the following result (the assumptions on the pseudogroup and the action 
will be explained and justified afterwards; they are required for all objects to be well-defined); we work over $\R$ or $\C$.

 \begin{theorem}\label{Th1}
Consider an algebraic Lie pseudogroup $G$ acting transitively on a manifold $M$ with an induced action on the space of jets 
$J^\infty$, which eventually (from some jet-level) is almost locally free, or on a compatible  
differential equation $\E$, considered as an invariant (profinite, co-filtered) submanifold therein.

Let $\mathcal{A}$ be the field of global rational differential invariants and let $\mathcal{I}$ be the field of
local differential invariants obtained by the method of moving frames. Then $\mathcal{I}$ is an algebraic extension 
of $\mathcal{A}$. 
 \end{theorem}
 
For finite-dimensional Lie groups the action eventually becomes locally free \cite{AO} 
(the assumption of analyticity in \cite{AO} follows from ours\footnote{Note however an example in 
\cite[\S3.3]{KS+} where the prolongation of a Lie group $G$ action on an equation $\E$ is not free.}).
For infinite-dimensional Lie pseudogroups the condition that the action eventually becomes free is non-vacuous,
though in most interesting examples it is satisfied \cite{OP,OPV}.

To the other direction, we often can realize $\mathcal{A}$ as the fixed point set of the action of a finite group $\Gamma$
of automorphisms on $\mathcal{I}$, which we will call the {\em deck group}. 
Over $\C$ this $\Gamma$ is the Galois group of the extension,
and it often can be chosen the stabilizer of a quasi-section $\Sigma$, which is the Zariski closure 
of the local section $\Sigma'$ used in the normalization of the moving frame. 

Over $\R$ the stabilizer subgroup of the section $\Sigma$ may be a proper subgroup of the deck group 
$\Gamma_{\!{\sf R}}\subset\Gamma$, as we will see in examples. 
The fixed field $\mathcal{J}$ of the $\Gamma_{\!{\sf R}}$ action on $\mathcal{I}$ separates generic orbits of $G$, yet it may 
contain radicals, see Remark \ref{rk2} (a possibility of roots in global differential invariants was mentioned in \cite{KL2}). 
The corresponding field extensions are nested:
 \begin{equation}\label{AJI}
\mathcal{A}\subset\mathcal{J}\subset\mathcal{I}. 
 \end{equation}

By \cite[Theorem 1]{KL2} singularities of algebraic pseudogroup actions stabilize in finite order:
for some $\ell$ and Zariski open dense subset $\E_o^\ell\subset\E^\ell$ the action of $G$ on 
$\pi_{\ell+k,\ell}^{-1}(\E_o^\ell)$ admits a geometric quotient.
The following adapts the setup of Theorem \ref{Th1}. We will give details on quasi-sections and
$\Gamma$-sections in the next sections.

 \begin{theorem}\label{Th2}
Let $\ell$ be the order of stabilization of singularities of the $G$ action on $\E$. Assume that stabilizers of generic points
of $\E^\ell$ belong to the same conjugacy class in $G$. Then there exists a quasi-section that is a $\Gamma$-section 
of the action for a finite subgroup $\Gamma\subset G$. (In particular, this is so if $G$ has one open orbit in $\E^\ell$.)

More generally, provided a $\Gamma$-section on $\E^\ell$ exists, the group $\Gamma$ can be chosen as its stabilizer. 
In this case $\mathcal{A}$ is the fixed point set of the action of $\Gamma$ by automorphisms of $\mathcal{I}$.
 \end{theorem}

Note that (except for finite type equations $\E$) the prolongation $\E^\infty$ is infinite-dimensional,
so the finiteness of $\Gamma$ as well as that for $[\mathcal{I}:\mathcal{A}]$ does not follow from algebraic theory 
(also note that both fields of invariants have infinite transcendence degree) but 
is a corollary of stabilization of singularities. 
 
The above theorem yields an effective method to compute the algebra 
of global differential invariants from local differential invariants, also known as Cartan invariants, in general.
The main result can be summarized as the following formula for the algebra of global differential invariants
(where $\mathcal{R}$ is the field of rational functions on $\E$)
 \begin{equation}\label{AIG}
\mathcal{A}=\mathcal{R}^G=\mathcal{I}^\Gamma.
 \end{equation}
We introduce the necessary background in the next sections, and then we prove the main results.
 
Afterwards we consider several examples that illustrate the theorems. We start with differential invariants
of un-parametrized curves in the Euclidean space (that can be also considered parametrized by a natural parameter),
then we consider curves in flat conformal and projective spaces. Differential invariants for these problems are
known, and one can observe radicals in their expressions. We explain all those radicals through
the result of Theorem \ref{Th1}, compute the deck group $\Gamma$ and then express global differential invariants. 
Real curves in generalized flag varieties $G/P$ through the method of moving frames were discussed in \cite{DZ},
these possess the complexification as complex curves in complex flag varieties $G_\C/P_\C$.

Jets of curves in various Klein geometries are among the simplest possible infinite-dimensional spaces. One can also 
consider Lie groups acting on higher-dimensional submanifolds. A classical example is given by surfaces in the
Euclidean 3-space. The principal curvatures $k_1,k_2$ (differential invariants of order 2) are eigenvalues of
the Weingarten (shape) operator. These contain a square root, which can be explained through Theorem \ref{Th2},
as the swap operator $k_1\leftrightarrow k_2$ generates $\Gamma'=\Z_2$.
The global invariants obtained from those by symmetrization are: 
the Gaussian curvature $K=k_1k_2$ and the mean curvature $H=(k_1+k_2)/2$.
The invariant derivations arise from normalized principal directions $e_1,e_2$, 
but those are defined up to $\pm$ and swap, whence $\Gamma'$ extends to 
$\Gamma=\Z_2^{\times4}$ for $G=O(3)\ltimes\R^3$ (or $\Gamma^\dagger=\Z_2^{\times3}$ for 
orientation preserving motions $G^\dagger=SO(3)\ltimes\R^3$);
global invariants are obtained by $\Gamma$-averaging $e_i(k_j)$ and further iterated derivatives.

We also consider an example of infinite-dimensional Lie pseudogroup acting on surfaces, for which the theorem applies.
We finish with an outlook, comparing our result with the Chevalley theorem on Casimir invariants. 
Both results allow reducing the computation of the algebra of invariants to averaging 
with respect to a finite group, which is computationally feasible. 

We will discuss (pseudogroup) actions exclusively over the fields $\R$ or $\C$.
While algebraic groups are best dealt with over $\C$, our computations are over $\R$ to refer to the familiar context.
Complexifications of those computations are straightforward, and we'll make them explicit in examples. 
However the complex and real version may differ in the corresponding deck group $\Gamma$, as the real 
extension may not be Galois, see Remark \ref{rk3}.
Also, in general, different real versions may vary in the manner the invariants are generated.

\section{Moving frames via quasi-sections}\label{S2}

We begin with classical invariants of an action of a finite-dimensional Lie group $G$ on a finite-dimensional manifold $M$.
In general, such an action is assumed smooth or holomorphic.
An equivariant moving frame may be defined as a $G$-equivariant map from $M$ (or rather an open subset $U\subset M$)
to $G$, see \cite{Gre,Gri,FO1,FO2}. 
Since the action of $G$ on itself is free, such should be also the action of $G$ on $M$, 
as a necessary condition for the existence of a moving frame.
Thus for the case, when a stabilizer $G_x$ of a point $x\in M$ is a Lie subgroup in $G$ of positive dimension,
the method of moving frames is not applicable (even locally).

However, even in the simplest case when the action is free, the section of the orbit foliation, 
required to construct a moving frame, may not exist. In fact, this case corresponds to a principal $G$-bundle $M\to B$, 
for which the obstructions to the existence of global sections belong to $H^k\bigr(B,\pi_{k-1}(G)\bigl)$ for various $k>0$.

A more involved case is that of locally free actions (that is, when the stabilizer $G_x$ of every point $x\in M$ is discrete in $G$)
and almost locally free actions (when the same holds for a generic point $x\in M$). The latter case is reduced to the former
by a removal of singularities. To deal with the former, we restrict (here and in the following) to the {\em algebraic actions}
$G\times M\to M$.

 \begin{definition}
An irreducible subvariety $\Sigma\subset M$ is called a rational section, resp.\ quasisection, of the action of $G$
if there exists a (Zariski) open nonempty subset $\Sigma_o\subset\Sigma$ with open dense $M_o=G\cdot\Sigma_o\subset M$
such that any orbit of the induced action of $G$ on $M_o$ intersects $\Sigma_o$ exactly once, resp.\ at finitely many points. 
 \end{definition}
 
If $\nu:M\to N$ is a rational quotient of the action of $G$, stemming from the Rosenlicht theorem \cite{R}, then 
$\rho:\Sigma\to N$ is a rational covering. More generally, representing the quasi-section by a rational mapping
$\sigma:Z\to M$ with $\Sigma=\sigma(Z)$ we get a rational covering $\rho\circ\sigma:Z\to N$, 
which may be chosen to be Galois\footnote{While this choice of $Z$ is not minimal (not a rational embedding) 
it uniformizes the source $Z$ up to birational equivalence.}.
 
It is seldom that (global) sections exist (see \cite[\S2.5]{PV} where obstructions in Galois cohomology are identified).
Customary, in the moving frame method, the manifold $M$ is therefore restricted to an open set $U\subset M$ in 
set-theoretic (smooth or analytic, but not Zariski) topology, and sometimes $\Sigma$ is also restricted to $\Sigma'$, 
such that $G\cdot\Sigma'\supset U$. In this case $G$ becomes a local Lie group (or groupoid) on $U$ and invariant
functions are identified with elements of proper functional space $F(\Sigma)$: due to unicity of intersection there is 
a unique extension of a function on the section to an invariant function; this is the {\em invariantization} procedure.

By \cite[Proposition 2.7]{PV} the action always admits a quasi-section $\Sigma$, and it transversally intersects 
almost every orbit $G\cdot x$, with degree $0<|(G\cdot x)\cap\Sigma|<\infty$. One then takes $\Sigma'$ to be the
part of $\Sigma$ with transversal intersections, and $U$ a ``collar'' of $\Sigma'$ consisting of parts of $G$-orbits
that do not intersect. This $U$ is not $G$-invariant, but admits a local geometric quotient. 

\medskip

Next, consider an action of a Lie group $G$ on an infinite-dimensional manifold $\mathfrak{M}$, more precisely
a diffiety in the terminology of \cite{V}. This latter may be the space of jets $J^\infty(\R,M)$ of parametrized curves in $M$ 
($\R$ is a sample source; in periodic case it is $S^1$; for complex curves it is $\C$, or $S^2=\C P^1$ or higher genus curve)
or the space of jets of unparametrized curves $J^\infty(M,1)$, higher dimensional submanifolds $J^\infty(M,k)$,
jets of functions $J^\infty(M,\R)$, sections of bundles $J^\infty(M,E)$, etc; see \cite{KL1,O} for a general setup. 
Also $\mathfrak{M}$ may be a differential equation viewed geometrically as a (co-filtered) submanifold $\E$ in one of these jet spaces. In general, a diffiety is a scheme based on equations $\E$ as charts \cite{V}.

It is important that $\mathfrak{M}=\lim\mathfrak{M}^k$ is a projective limit of finite-dimensional submanifolds,
in the simplest case $\mathfrak{M}^k=J^k$. 
A Lie pseudogroup $G$ may be considered as a diffiety of its own, compatible with the group operations \cite{KL1}.
The action of $G$ on the base $M$ is naturally prolonged to the action on $\mathfrak{M}$. 
Following \cite{KL2} we assume that the action is transitive on $M$ and algebraic on $\mathfrak{M}$,
the latter means algebraicity of the action in fibers of the projection $\mathfrak{M}^k\to M$ for every $k$. 
The fibers over a point $x\in M$ is $\mathfrak{M}^k_x$ and the stabilizer group $G^{k+r}_x$ (algebraically) acts on it, 
if the order of the action is $r$.

If the pseudogroup is infinite-dimensional, the action of $G^{k+r+1}_x$ on $\mathfrak{M}^k_x$ is non-effective 
and hence non-free for any $k$. Following \cite{OP} the action is called (almost locally) free if such is the action 
of $G^{k+r}_x$ on $\mathfrak{M}^k_x$ for every sufficiently large $k$ and each $x\in M$.
In this case the quasi-sections $\Sigma^k_x$ can be chosen compatibly: $\pi_{k+1,k}(\Sigma^{k+1}_x)=\Sigma^k_x$
for all $k$ (see \cite{OP} for details and \cite{OPV} for comparison to the Cartan approach). 
Due to this, the collection $\Sigma^k$ over $M$ can be considered as a diffiety itself.

In the moving frame approach, this $\Sigma$ is considered as a quotient equation: functions on $\Sigma$
uniquely extend to differential invariants \cite{FO2}, so their algebra is identified with 
 \begin{equation}\label{IFS}
\mathcal{I}=F(\Sigma),
 \end{equation}
where $F=C^\infty$ in the smooth context, however in the algebraic context it is the field $\mathcal{R}$ of rational functions. 
Note that if we use semi-algebraic $\Sigma'$ for normalization, we still get $\mathcal{I}\simeq F(\Sigma')$, as rational
functions uniquely extend from an open set (in set-theoretic $C^0$ topology) to its Zariski closure. 

Formula \eqref{IFS} represents differential invariants only locally. 
Global differential invariants emerge only from functions on $\Sigma$ that take the same values at different intersections
of an orbit with the quasi-section and a basis of such can be chosen rational.
This gives the global quotient equation as a diffiety through its field of functions up to birational equivalence. 
By \cite[Theorem 2]{KL2} this field $\mathcal{A}$ is finitely generated by a finite number of differential invariants and 
a finite number of invariant derivations.

 \begin{proof}[Proof of Theorem \ref{Th1}]
If $G$ acts algebraically on $M$ and $\Sigma$ is a quasi-transversal then,
by the proof of \cite[Proposition 2.7]{PV} and by \cite[Theorem 2.17]{HK2}, 
the field of rational functions $\C(\Sigma)$ is a finite extension of the field $\C(M)^G$ of rational invariants,
with the extension degree $d$ equal to the number of intersections of generic orbit $G\cdot x$ with $\Sigma$
(by \cite[Proposition 3.2]{HK1} this is equal to $d$ on a Zariski open set $\Sigma_o$ over $\C$).

While these references restrict to an algebraic closed base field, i.e.\ $\C$ in our case, the conclusions hold true for $\R$ as well
(while most sources assume algebraic closure of the base field, Rosenlicht's argument \cite{R} applies to the field of definition 
and can be made constructive \cite{HK1,Kem}). 
It is however easier to argument it as follows. Given a real algebraic action, consider its complexification. Then
$\mathcal{I}_\C$ is an algebraic extension of $\mathcal{A}$. But since $\mathcal{I}$ is an intermediate field, it is
a finite extension and hence an algebraic extension of $\mathcal{A}$.

Now apply this to the action of the stabilizer $G_x^{k+r}$ acting on $\E^k_x$ for arbitrary $x\in M$,
taking into account that the invariants of $G^\infty$-action on $\E^\infty$ are in one-to-one with
the invariants of $G^\infty_x$-action on $\E^\infty_x$. Since this action
is algebraic we conclude that $\mathcal{I}^k$ is an finite extension of $\mathcal{A}^k$ of some degree $d_k$ 
growing with $k$ up to $k=\ell$. 
From this stabilization jet-level the action is free affine \cite{OP,KL2}, and so (now infinite) extensions 
of $\mathcal{I}^{\ell+j}$ of $\mathcal{I}^\ell$ and $\mathcal{A}^{\ell+j}$ of $\mathcal{A}^\ell$ are equivalent.
Thus the extension $\mathcal{I}=\mathcal{I}^{\infty}$ of $\mathcal{A}=\mathcal{A}^\infty$ has finite degree $d_\ell$.
 \end{proof}

\section{$\Gamma$-sections and differential invariants}\label{S3}

Now let us discuss the deck group $\Gamma$. We will assume $\Sigma$ to be algebraic Zariski connected.

 \begin{prop}\label{Prop1}
Over $\C$ the field $\mathcal{I}$ is a Galois extension of $\mathcal{A}$. 
 \end{prop}
 
 \begin{proof}
On a finite jet-level this follows from the fact that the algebraic extension $\mathcal{I}$ of $\mathcal{A}$
is normal: it adjoins to $\mathcal{I}$ all roots of irreducible polynomials over $\E$, responsible for the extension
$\mathcal{I}/\mathcal{A}$. 

Indeed, a small neighborhood $U=U_1\subset\Sigma$ (in $C^0$ topology) satisfies 
$(G\cdot U)\cap\Sigma=\cup_{i=1}^d U_i$ with identifications $\sigma_i:U\to U_i$ through $G$-orbits foliation;
here $d$ is the degree of quasi-secton $\Sigma$, we refer to \cite[Definition 2.5]{HK2} for algebraic constructive description.
A rational function $f$ on $U$ uniquely extends to $\Sigma$ 
and we similarly get rational functions $f_i=\sigma_i^*f$ on $\Sigma$.
Transformations $\sigma_i$ shuffle branches $f_i$ over $U_j$ and hence generate a subgroup $\Gamma\subset S_d$
acting on $\mathcal{I}$ with averaging $f\mapsto\bar{f}=\frac1d\sum_1^d \sigma_i^*f_i$.
Thus the minimal polynomial $P(\zeta)=\prod(\zeta-f_i)$ for the extension $f$ of $\bar{f}$ contains all roots in $\mathcal{I}$.

In the language of moving frames \cite{FO2} this is given by algebraic invariantization \cite[Proposition 2.19]{HK2}:
the Galois group of the extension permutes the set of replacement invariants $\xi_1,\dots,\xi_d$ 
corresponding to a choice of quasi-section
$\sigma:(\xi_1,\dots,\xi_d)\mapsto(\xi_{\sigma(1)},\dots,\xi_{\sigma(d)})$, $\sigma\in\Gamma\subset S_d$. 
Here $\xi_i$ are local invariants
inducing coordinates $z_i$ on $\Sigma$ and invariantization $\C[z]_\Sigma\to\overline{\C(z)}^G$ takes the form
$r(z)\mapsto r(\xi)$.

With this we addressed the algebraic setup, which can apply to the action of $G_x^{k+r}$ on $\E_x^k$ for
$k\leq\ell$. As in the previous proof, from the jet-level $\ell$ the additional extensions of $\mathcal{I}$ 
over $\mathcal{A}$ are trivial. Thus the extension is normal. Since the ground field is of characteristic zero, it is separable.
Hence by the Artin theorem the extension is Galois, given by a group $\Gamma$ so that \eqref{AIG} holds.
 \end{proof}

We conclude that the Galois group $\Gamma$ is the extension group of $\mathcal{A}$ to $\mathcal{I}$,
and it may serve as the deck group of the corresponding regular cover over the quotient equation
$\Sigma\to\mathfrak{N}:=\E/G$, where both are considered diffieties (towers of fiber bundles). 
Over $\R$ the intermediate extension in \eqref{AJI} may fail to be Galois because even though $\Sigma$ is transversal
to generic $G$-orbits, some intersection may disappear into imaginary domain and the intersection index of $\Sigma$
with the orbit may be less than the degree of the quasi-section.

The deck group $\Gamma$ of Proposition \ref{Prop1} may not be a subgroup of $G$ acting on the quasi-section $\Sigma$. 
Since quasi-sections are somewhat generic, one may not expect $\Gamma$ to be their stabilizer, in general.
There is, however, a more restrictive notion of section relative to a subgroup, introduced in \cite{Kat}:

 \begin{definition}
For algebraic subgroup $\Gamma\subset G$, an unmixed\footnote{That means, all components have the same dimensions.} 
subvariety $\Sigma\subset M$ is called a $\Gamma$-section if there exist a dense open subset $\Sigma_o\subset\Sigma$
such that (i) $\overline{G\cdot\Sigma}=M$, (ii) $\Gamma\cdot\Sigma=\Sigma$, 
(iii) $g\cdot\Sigma_o\cap\Sigma_o\neq\emptyset$
$\Rightarrow$ $g\in\Gamma$. 
 \end{definition}

Under the condition in the definition, the group $\Gamma$ is the stabilizer of $\Sigma$.
This implies right away the following expression for rational invariants
 \begin{equation}\label{CMGG}
\C(M)^G=\C(\Sigma)^\Gamma. 
 \end{equation}
There is the following algebraic criterion \cite{Kat,PV} for the existence of a $\Gamma$-section.

 \begin{prop}\label{KatPV}
Suppose $M_o\subset M$ is a nonempty (Zariski) open subset of $M$ such that $\forall x\in M_o$
the stabilizer $G_x\stackrel{conj}\simeq H\subset G$ for some fixed subgroup $H$. Then $\Sigma$, the union of 
maximal-dimension components of $M^H_o=\{x\in M_o:Hx=x\}$, is the $\Gamma$-section for $\Gamma=N(H)$.
 \end{prop}

For differential invariants with the stabilization of singularities on jet-level $\ell$ we conclude:
 \begin{prop}
Suppose for all $k\leq\ell$ the stabilizers $\op{Stab}(x_k)\subset G_x^{k+r}$ are conjugated for generic 
$x_k\in\pi_k^{-1}(x)\subset\E_x^k$, $x\in M$. Then the action of $G$ on $\E$ has $\Gamma$-section for $\Gamma=\Gamma^\ell$.
 \end{prop}
 
 \begin{proof}
Fix $x\in M$ (by transitivity of $G$-action on the base, its choice is irrelevant). By the algebraic result of Proposition 
\ref{KatPV}, for each $k\leq\ell$, we can construct $\Gamma^k$-section $\Sigma^k\subset\E^k_x$.
These sections are bundles over each other: $\Sigma^{k+1}\to\Sigma^k$ and the groups are successive left
extensions $\Gamma^{k+1}\twoheadrightarrow\Gamma^k$.

By stabilization of singularities \cite{KL2}[Theorem 1], starting from a jet-level $\ell$, the action of $G^{k+r+1}_x$
is an affine extension of the action of $G^{k+r}_x$ on the affine fiber bundle 
 $$
\E^{k+1}_x\setminus\pi^{-1}_{k+1,\ell}(S_\ell)\to\E^k_x\setminus\pi^{-1}_{k,\ell}(S_\ell),\quad k\geq\ell,
 $$ 
where $S_\ell$ is the singularity locus (and these fibers allow geometric quotient).
Thus, for $k\geq\ell$ the action is free affine with differential invariants affine by the highest jets, 
so the action in $\pi_{k+1,k}^{-1}(x_k)$, for $x_k$ outside singularities, admits a global section. 
Thus we can extend $\Sigma^\ell$ to $\Sigma=\Sigma^\infty$ through a sequence of affine bundles 
$\Sigma^{k+1}_o\to\Sigma^k_o$ with $\Sigma^k=\overline{\Sigma^k_o}$,
and $\Gamma=\lim\Gamma_k$ stabilizes; in fact $\Gamma=\Gamma^\ell$.
 \end{proof}
 
It follows that if the stabilization jet-level $\ell$ does not exceed the minimal jet-order of absolute differential invariants, then 
$G$-action admits a $\Gamma$-section over $\C$. Indeed, in this case there is a unique open orbit of $G$ action in $\E^\ell$.
Over $\R$ there may be several open orbits with different stabilizers $H$, so the stabilizer of the relative section
will be smaller than the complex deck group $\Gamma$. 

 \begin{proof}[Proof of Theorem \ref{Th2}]
The action of $G$ is transitive on the base $M$, so $G$-invariants in $\E$ correspond to invariants of the stabilizer $G_x$ 
in the fiber $\E_x$ for any $x\in M$. This latter action is algebraic and we can obtain a quasi-section $\Sigma_x\subset\E_x$.
If it is a $\Gamma$-section, then $\Gamma$ is precisely the stabilizer of $\Sigma$ in $G$. Differential invariants correspond
to rational functions on $\Sigma$ invariant with respect to $\Gamma$. Thus, the field of rational functions $\C(\Sigma)$ 
is an extension of the field $\C(\E)^G$ by this group $\Gamma$, which fixes $\mathcal{A}$. 
 %
 \end{proof}

For the base field $\R$, denote by $\Gamma^k_{\!{\sf R}}$ the stabilizer ($G$-automorphism subgroup) 
of $\Sigma^k_x$ in $G^{k+r}_x$ (this group conjugacy class is independent of $x\in M$) and by 
$\Gamma_{\!{\sf R}}=\lim_{k\to\infty}\Gamma^k_{\!{\sf R}}$ 
its stable projective limit. This group may be a proper subgroup of the deck group $\Gamma$ 
because some of its components may not intersect the real slice,
and the field of global invariants $\mathcal{J}=\mathcal{I}^{\Gamma_{\!{\sf R}}}$ may not be rational. 
Yet the field of global rational invariants is a finite degree subfield 
$\mathcal{A}=\mathcal{J}^{\Gamma/\Gamma_{\!{\sf R}}}$ (still separating orbits)
and so $\mathcal{A}$ is again an algebraic extension of $\mathcal{I}$.


\section{Frenet-Serret formulae revisited}\label{S4}

Let $x(t)$ be a smooth curve in $\R^n$ that is nondegenerate, in the following sense: its first $n$ derivatives are 
linearly independent\footnote{This can be relaxed to the condition that $(n-1)$ derivatives are linearly independent
if we fix a volume form $\Omega$, by omitting $x^{(n)}$ and using the formula 
$\Omega(e_1,\dots,e_{n-1},\cdot)=g(\cdot,e_n)$ instead.}. 
Equip $\R^n$ with the Euclidean metric $g_{ij}=\delta_{ij}$ (where the latter symbol is the Kronecker delta).
Then the Gram–Schmidt process applied to $x'(t),x''(t),\dots,x^{(n)}(t)$ yields an orthonormal basis,
called the Frenet–Serret frame $e_1(t),\dots,e_n(t)$. 
This process is split into orthogonalization
 $$
\bar{e}_i(t)= x^{(i)}(t)-\sum_{j=1}^{i-1}\frac{\langle r^{(i)}(t),\bar{e}_j(t)\rangle}{\langle \bar{e}_j(t),\bar{e}_j(t)\rangle}
 $$
and the following normalization $e_i(t)=\frac{\bar{e}_i(t)}{\|\bar{e}_i(t)\|}$. The former is clearly algebraic, while 
the latter is not, because $\|\bar{e}_i(t)\|=\sqrt{\langle\bar{e}_i(t),\bar{e}_i(t)\rangle}$ is multi-valued over $\C$
(and is not defined everywhere over $\R$ if the metric $g_{ij}$ is Minkowski or has another signature).

Denoting $E(t)=[e_1(t)|\dots|e_n(t)]$ the corresponding $n\times n$ matrix, the map $t\mapsto (E(t),x(t))\in G=O(n)\ltimes\R^n$
is a classical example of the moving frame (usually presented through exterior differential systems \cite{C1}, but we will use jets). 
Here $\mathfrak{M}^k\subset J^k(\R,\R^n)$ is the space of $k$-jets of nondegenerate curves, 
and the correspondence $t\mapsto\bigl(x(t),x'(t),x''(t),\dots,x^{(n)}(t)\bigr)$ defines a map $\mathfrak{M}^n\to F(\R^n)$,
where the space of frames $F(\R^n)\subset J^1_0(\R^n,\R^n)$ is defined as the set of 
1-jets of local diffeomorphisms of $(\R^n,0)$.

Both embeddings above are Zariski open. For any Riemannian manifold $M$,
the Gram--Schmidt process is the projection $F(M)\to FO(M)$ to orthonormal frames.
The composition yields a $G$-equivariant map $\mathfrak{M}\to FO(\R^n)=G$; 
moreover this eliminates the dependence on parametrization of curves.

Alternatively, we can choose the cross-section $\Sigma_+$, defining the moving frame, by $x(0)=0$, 
$x'(0)=e_1$, $x''(0)\in\R_+\!\cdot e_2$, $x'''(0)\in\R_+\!\cdot e_3\!\!\mod\R^1(e_2)$, 
$x^{\text{\sl iv}}(0)\in\R_+\!\cdot e_4\!\!\mod\R^2(e_2,e_3)$, \dots,
$x^{(n)}(0)\in\R_+\!\cdot e_n\mod\R^{n-2}(e_2,\dots,e_{n-1})$, 
where $\{e_i\}_{i=1}^n$ is the standard orthonormal basis of $\R^n$. 
Here we consider un-parametrized curves (or rather adjust parametrization to satisfy the conditions). 
This in turn defines an equivariant moving frame, as above. 

The generalized curvatures for $x(t)$ are defined as
 $$
\kappa_i(t)= \frac{\langle e'_i(t),e_{i+1}(t)\rangle}{\|x'(t)\|}\qquad (1\leq i<n).
 $$
These are considered as a basis of differential invariants of un-parametrized curves in the Euclidean space:
 $$
\mathcal{I}=\langle\kappa_i,\kappa_{is},\kappa_{iss},\dots\rangle =\langle\kappa_i|D_s\rangle,
 $$
where $D_s=\tfrac{d}{ds}$ is the derivation by natural parameter and $\kappa_{is}=D_s(\kappa_i)$, etc.

However these functions are not $G$-invariant. Indeed, the orthogonal transformation $\sigma_i:e_j\mapsto(-1)^{\delta_{ij}}e_j$ results in 
 $$
\kappa_j\mapsto(-1)^{\delta_{1i}+\delta_{ij}+\delta_{i,j+1}}\kappa_j.
 $$  
Thus $\kappa_i^2$ are global invariants. The derivation $D_s$, 
corresponding to $e_1$, is also non-invariant, but $\nabla=\kappa_1\kappa_{1s}D_s$ is invariant and we get 
(the last portion of generators is an invariant derivation)
 $$
\mathcal{A}=\langle\kappa_1^2,\dots,\kappa_{n-1}^2,\kappa_{1s}^2,
\kappa_1\kappa_2\kappa_{1s}\kappa_{2s},\dots,\kappa_1\kappa_{n-1}\kappa_{1s}\kappa_{n-1,s},\dots\rangle =\langle\kappa_i^2|\nabla\rangle. 
 $$
The local transversal $\Sigma_+$ is semi-algebraic, and we can complete it to algebraic quasi-section $\Sigma$:
$x(0)=0$, $x'(0)=e_1$, $x^{(k)}(0)\in\op{span}(e_2,\dots,e_k)$, $1<k<n$.
Thus we conclude \framebox{$\Gamma=\Z_2^n$} and
 $$
[\mathcal{I}:\mathcal{A}]=2^n. 
 $$
 
Next, consider the (connected component) subgroup $G^\dagger=SO(n)\ltimes\R^n\subset G$. 
Then the moving frame is obtained by choosing the sign $\pm e_n$ to meet the orientation of $\R^n$; 
in the complex case the equivalent form is $\det[e_1|\dots|e_n]=1$.
The corresponding frame bundle $FSO(M)$ is a principal $SO(n)$ bundle over $M$.

This fixes one sign overall for the curvatures and hence we get \framebox{$\Gamma^\dagger=\Z_2^{n-1}$} generated by 
involutions $\sigma_{ij}=\sigma_i\sigma_j$, implying 
 $$
[\mathcal{I}:\mathcal{A}^\dagger]=2^{n-1}.
 $$
The generation of the fields of $G^\dagger$-invariants $\mathcal{A}^\dagger$ is more complicated and depends on dimension $n$:
for the first values of $n$ we have:
 \begin{gather*}
\mathcal{A}^\dagger_{n=2}=\langle\kappa_1^2|\kappa_1D_s\rangle,\\
\mathcal{A}^\dagger_{n=3}=\langle\kappa_1^2,\kappa_2|\kappa_1\kappa_{1s}D_s\rangle,\\
\mathcal{A}^\dagger_{n=4}=\langle\kappa_1^2,\kappa_2^2,\kappa_1\kappa_3|\kappa_1\kappa_{1s}D_s\rangle,\\
\mathcal{A}^\dagger_{n=5}=\langle\kappa_1^2,\kappa_2^2,\kappa_3^2,\kappa_4^2|\kappa_2\kappa_4D_s\rangle.
 \end{gather*}
The $\Gamma^\dagger$-invariant subalgebra $\mathcal{A}^\dagger_\kappa$ in $\R[\kappa_1,\dots,\kappa_{n-1}]$ 
generated by the curvatures (without derivations) exhibits the following 4-periodic behavior: 
for $n\equiv0,3\!\!\mod4$ it is generated by $\kappa_i^2$ ($1\leq i<n-1$) and also 
 \[
n=3: \kappa_2,\quad 
n=4: \kappa_1\kappa_3,\quad
n=7: \kappa_2\kappa_4\kappa_6,\quad
n=8: \kappa_1\kappa_3\kappa_5\kappa_7,\quad 
n=11: \kappa_2\kappa_4\kappa_6\kappa_8\kappa_{10},\quad\dots
 \]
For others, $n=2,5,6,9,10,\dots$ (that is, $n\equiv1,2\!\!\mod4$) it is generated by
$\kappa_i^2$ ($1\leq i<n$). We conclude:
 $$
[\mathcal{I}^\dagger_\kappa:\mathcal{A}^\dagger_\kappa]
=\left\{\begin{array}{ll}2^{n-2}, & n\equiv0,3\!\!\mod4\\ 2^{n-1}, & n\equiv1,2\!\!\mod4 \end{array}\right.
 $$

For $n\equiv1,2\!\!\mod4$ the invariant derivation has the form $\bar\nabla=h(\kappa_1,\dots,\kappa_{n-1})D_s$. 
For $n\equiv3,4\!\!\mod4$ the simplest the invariant derivation is $\nabla=\kappa_1\kappa_{1s}D_s$, but for
algebraic generation (without differentiation) of $\mathcal{A}^\dagger$ we have to add the generator $\kappa_{1s}^2$,
which explains universal form of $\Gamma^\dagger$.

\smallskip

{\it Complexification.} 
For real curves in flat metric spaces of different signature $\R^{p,q}$ or for complex curves in $\C^n$ 
with a $\C$-linear inner product, the local invariants $\kappa_i,\kappa_{is},\dots$ are applicable only in open 
(but not dense) sets $U$ (where the roots or their branches have sense and the length is nonzero) but the global invariants 
are given by the same formulae as for the Euclidean case, and $\Gamma$ is un-changed. 

\smallskip
 
{\it Example.}
To make the formulae more concrete in jet-notations, here are the curvature $\kappa=\kappa_1$ and the torsion $\tau=\kappa_2$
for the classical case of curves $\{y=y(x),z=z(x)\}$ in the Euclidean 3-space $\R^3(x,y,z)$:
 \begin{equation}\label{Eucl3D}
\kappa=\frac{\sqrt{(y_1z_2-z_1y_2)^2+y_2^2+z_2^2}}{(1+y_1^2+z_1^2)^{3/2}},\quad
\tau=\frac{y_2z_3-z_2y_3}{(y_1z_2-z_1y_2)^2+y_2^2+z_2^2}\ \text{ and }\ 
D_s =\frac1{\sqrt{1+y_1^2+z_1^2}}D_x.
 \end{equation}
The moving frame allows us to write the curve via local invariants (with $\kappa_s=D_s\kappa$, etc)
 $$
y= \kappa\,\frac{x^2}{2!} + \kappa_s\,\frac{x^3}{3!} + (\kappa_{ss} + 3\kappa^3 - \kappa\tau^2)\,\frac{x^4}{4!}+O(x^5),\quad
z= \kappa\tau\,\frac{x^3}{3!} + (\kappa\tau_s + 2\kappa_s\tau)\,\frac{x^4}{4!}+O(x^5). 
 $$

The field of rational differential invariants is generated by
 $$
\kappa^2,\quad \tau,\quad \nabla=\kappa\kappa_s D_s.
 $$
 
 \begin{remark}\label{rk1}
Curves in the Euclidean space can be parametrized by a natural parameter (defined up to shift).
Curves in other Klein geometries do not have such canonical parametrizations.
In conformal and projective geometry, discussed in the next two sections, they
possess a natural projective parameter, that is a preferred family of parametrizations
differing by M\"obius transformations. Such projective parameter exists for curves
in various homogeneous geometries \cite{D} with a canonical moving frame along them \cite{DZ}.
 \end{remark}

\section{Un-parametrized curves in the conformal sphere}\label{S5}

Recall that the conformal sphere $\mathcal{Q}^n$ is the projectivization (sometimes ray projectivization)
of the null cone in the space $\R^{p+1,q+1}$.
For definiteness we restrict to the Euclidean signature, when the tractor space is Lorentzian $\R^{n+1,1}$, and the null
cone is given by $\sum_1^nx_i^2-2x_0x_{n+1}=0$.

Here the situation is similar to the previous: local differential invariants contain radicals, see \cite{MB,MMR} and references therein. 
The conformal version of Frenet-Serret formulae is (see \cite{MMR}; here the dot means the derivative wrt conformal parameter $s$):
 \begin{gather*}
\dot{e}_0=  e_1,\quad \dot{e}_1=\mu_1 e_0+e_{n+1},\quad \dot{e}_2=e_0+\mu_2 e_3,\\
\dot{e}_k=-\mu_{k-1}e_{k-1}+\mu_ke_{k+1}\ (3\leq k<n),\quad \dot{e}_n=-\mu_{n-1}e_{n-1},\quad 
\dot{e}_{n+1}=\mu_1 e_1+e_2. 
 \end{gather*}
The frame is determined up to rescaling. Indeed, consider the grading 
$\g=\mathfrak{so}(n+1,1)=\g_{-1}\oplus\g_0\oplus\g_1$ corresponding to matrix representation
in block form with blocks of sizes $(1,n,1)$:
 $$
\g=\begin{pmatrix} \g_0 & \g_1 & 0\\ \g_{-1} & \g_0 & \g_1\\ 0 & \g_{-1} & \g_0 \end{pmatrix}.
 $$
The stabilizer of a point in $\mathcal{Q}^n=G/P$ is the parabolic subgroup $P=G_0\cdot\exp\g_1$,
where $G_0=CO(n)$ and $\g_1=\R^n$ have general elements ($A\in O(n)$, $\lambda\in\R_\times$, $w\in\R^n$)
 $$
\delta_{\lambda,A}=\op{diag}(\lambda,A,\lambda^{-1})\in G_0,\quad 
\delta_w=\exp\begin{pmatrix}0 & -w^t & 0\\ 0 & 0 & w\\ 0 & 0 & 0 \end{pmatrix}\in\exp\g_1.
 $$
The element $\delta_{\lambda,A}\delta_w$ preserves the frame if $w=0$,
$A=\op{diag}(a_1,\dots,a_n)$, $a_i=\pm1$, and in addition $\lambda^2=a_1a_2$.
Thus $\lambda^4=1$, so in the real case $\lambda=\pm1$, and
consequently \framebox{$\Gamma_{\!{\sf R}}=\Z_2\times\Z_2^{n-1}=\Z_2^n$} with group parameters $(\lambda,a_2,\dots,a_n)$ since
$a_1=a_2$. It acts on local invariants as follows: $\mu_k\mapsto\tilde\mu_k$, where
 \begin{equation}\label{tildemu}
\tilde\mu_1=\lambda^2\mu_1=\mu_1,\ \tilde\mu_2=\lambda^{-1}a_1a_2a_3\mu_2=\lambda a_3\mu_2,\ 
\tilde\mu_k=\lambda^{-1}a_1a_ka_{k+1}\mu_k=\lambda a_2a_ka_{k+1}\mu_k\ (2<k<n).
 \end{equation}
In addition, since $\tilde{\p}_s=\lambda^{-1}a_1\p_s=\lambda a_2\p_s$ with invariant derivations 
$\bar\nabla=\mu_{1s}\p_s$, so we have
 \begin{equation}\label{IJ}
\mathcal{I}=\langle\mu_1,\dots,\mu_{n-1}|\p_s\rangle,\qquad
\mathcal{J}=\langle\mu_1,\mu_2^2,\dots,\mu_{n-1}^2|\bar\nabla\rangle
 \end{equation}
(where $\mathcal{J}$ is the algebra of global real invariants) and we conclude
 $$
[\mathcal{I}:\mathcal{J}]=2^n. 
 $$

{\it Complexification.} 
In the complex version of the same problem, the equation $\lambda^4=1$ has four different roots in $\C$, 
so $\lambda\in\Z_4$ (multiplicative group of fourth roots of unity), and consequently 
\framebox{$\Gamma=\Z_4\times\Z_2^{n-1}$} with group parameters $(\lambda,a_2,\dots,a_n)$ since
$a_1=\lambda^2a_2$. It acts on local invariants as before \eqref{tildemu} but now 
$\tilde\mu_1=\lambda^2\mu_1=\pm\mu_1$ and the invariant derivation is
$\nabla=\mu_1\mu_{1s}\p_s$, so we have
 \begin{equation}\label{IA}
\mathcal{I}=\langle\mu_1,\dots,\mu_{n-1}|\p_s\rangle,\qquad
\mathcal{A}=\langle\mu_1^2,\mu_1\mu_2^2,\dots,\mu_1\mu_{n-1}^2|\nabla\rangle,
 \end{equation}
and we conclude
 $$
[\mathcal{I}:\mathcal{A}]=2^{n+1}. 
 $$  
 
 \begin{remark}\label{rk2}
Note that $\mathcal{J}$ from \eqref{IJ} in the real case is the algebra of global differential invariants, but they contain roots 
(in both generators $\mu_1$ and $\bar\nabla$). The algebra $\mathcal{A}$ of real rational differential invariants is given 
by the same formula \eqref{IA} as in the complex case. Both real algebras $\mathcal{J}$ and $\mathcal{A}$ separate 
generic $G$-orbits: the former by construction, and the latter by the real Rosenlicht's theorem \cite{R} 
(and its Lie-Tresse version \cite{KL2}). Indeed, the main difference is in the generator $\mu_1=T^2$ in the notation
\eqref{def_conf3D_inv_der} below. 
One notes that it is nonnegative, so taking another square $\mu_1^2=T^4$ in \eqref{TQD} is invertible over $\R$. 
 \end{remark}

To make the situation more transparent, let us again give concrete formulae for the simplest case $n=3$.
The generators of $\mathcal{I}$ are the following two differential invariants and an invariant derivation,
expressed in terms of Euclidean curvature $\kappa$, torsion $\tau$ \eqref{Eucl3D}, and natural parameter $ds$:
 \begin{equation}\label{def_conf3D_inv_der}
\begin{split}
     Q & =\frac{4\nu\nu_{ss}-4\kappa^2\nu^2-5\nu_s^2}{8\nu^3},\qquad 
T=\frac{2\kappa_s^2\tau+\kappa^2\tau^3+\kappa\kappa_s\tau_s-\kappa\kappa_{ss}\tau}{\nu^{5/2}},\\
\nu & =\sqrt{\kappa^2\tau^2+\kappa_s^2},\qquad \omega=\sqrt{\nu}ds,\quad D_\omega=\nu^{-1/2}D_s.
\end{split}
 \end{equation}

The first absolute invariant comes in order 4 and is called the conformal torsion $T$. In order 5 we get two more invariants:
conformal curvature $Q$ and the derivative of torsion $T_\omega=D_\omega T$. Order 6 yields 
two more invariants $Q_\omega$, $T_{\omega\omega}$, etc. The relative invariant $\nu$ is 
called the \textit{conformal arclength}. 

The quasi-section $\Sigma$ defining the moving frame is 
$x(0)=0$, $x'(0)=e_1$, $x''(0)=0$, $x'''(0)=e_2$, $x^{iv}(0)\propto e_3$.
We can again write a normal form of the curve in conformal $\R^3(x,y,z)$ 
with coordinate $x$ as a parameter:
 $$
y=\frac{x^3}{3!}+(2Q-T^2)\,\frac{x^5}{5!}+(2Q_\omega-3TT_\omega)\,\frac{x^6}{6!}+O(x^7),\ \ 
z=T\,\frac{x^4}{4!}+T_\omega\,\frac{x^5}{5!}+(T_{\omega\omega}-T^3+7TQ)\,\frac{x^6}{6!}+O(x^7).
 $$

The residual (deck) group $\Gamma$ is given by the transformation
$(x,y,z)\mapsto\Bigl(\sigma x,\sigma^{-1}y,\epsilon\sigma z\Bigr)$, where $\epsilon^2=1$, $\sigma^4=1$. 
With respect to this discrete freedom local invariants \eqref{def_conf3D_inv_der} change as follows:
 $$
T\mapsto\epsilon\sigma T,\qquad Q\mapsto\sigma^2Q,\qquad D_\omega\mapsto\sigma^{-1} D_\omega,
 $$
Note also that the Euclidean curvature and the natural parameter change $\kappa\mapsto\sigma^2\kappa$,
$\tfrac{d}{ds}\mapsto\sigma^2\tfrac{d}{ds}$, while $\tau$ is invariant. 
Thus the expressions containing even number of $\kappa$ and $s$-derivatives are invariant, and so the algebra $\mathcal{A}$
of global rational conformal invariants has the following generators:
 \begin{equation}\label{TQD}
T^4,\qquad QT^2,\qquad \nabla=TT_\omega D_\omega. 
 \end{equation}
The first invariant is $T^4$ of order 4, then come $Q^2$, $QT^2$ and $T_\omega^2$ of order 5, but we note that 
$Q^2=(QT^2)^2/T^4$ is derived. In order 6 we get $T^{-1}T_\omega Q_\omega$ and $TT_{\omega\omega}$, etc. 
The lowest order global invariant form is $T^{-1}T_\omega^{-1}\omega$.
All those are obtained as fixed point of the deck group \framebox{$\Gamma=\Z_4\times\Z_2$}.

\smallskip

It is instructive to see the action of the stabilizer $H$ of the point 
$p_2=\{x=0,y(0)=0,y'(0)=0,y''(0)=0,z(0)=0,z'(0)=0,z''(0)=0\}$ on
$\pi_{5,2}^{-1}(p_2)\subset J^5(\R_x^1,\R_{y,z}^2)\subset J^5(\mathcal{Q}^3,1)$ in jet-coordinates $(y_3,z_3,y_4,z_4,y_5,z_5)$. 
The subgroup is parametrized as follows
 $$
H=\left\{\begin{pmatrix}\lambda & 2\lambda b & 0  & -2\lambda b^2\\
0 & a & 0 & -2ab\\ 0 & 0 & A & 0\\ 0 & 0 & 0 & \lambda^{-1}\end{pmatrix}\,:\,
A=\begin{pmatrix}\alpha & \beta\\ \gamma & \delta\end{pmatrix}\in O(2),\, a=\det A,\, \lambda\in\R_\times,\, b\in\R\right\}
\subset SO(4,1)
 $$
and its action in local chart given by
 $$
(x,y,z)\mapsto
\left(\frac{a(x+b\ell^2)}{\lambda(1+2bx+b^2\ell^2)}, \frac{\alpha y+\beta z}{\lambda(1+2bx+b^2\ell^2)},
\frac{\gamma y+\delta z}{\lambda(1+2bx+b^2\ell^2)}\right),\ \text{ where }
\ell^2=x^2+y^2+z^2,
 $$
has algebraic prolongation to $\pi_{5,2}^{-1}(p_2)$.

Elimination of parameters leads to invariants $T^2$ and $Q$, which are global invariants given by 
 \begin{gather*}
T=\frac{y_3z_4- y_4z_3}{(y_3^2+z_3^2)^{5/4}},\quad
Q= \frac{(y_3^2+z_3^2)(y_3y_5+z_3z_5) + (z_3^2-\tfrac54y_3^2)y_4^2 - \tfrac92y_3z_3y_4z_4 + (y_3^2-\tfrac54z_3^2)z_4^2}
{2(y_3^2+z_3^2)^{5/2}}.
 \end{gather*}

\section{Un-parametrized curves in the projective space}\label{S6}

The last example of finite-dimensional Lie group action we consider is the projective group $G=PGL(n+1)$ acting on 
$\mathbb{P}^n$ (this can be considered over $\R$ or over $\C$) and we induce the action on unparametrized curves.

We start with curves in the projective plane ($n=2$) given by $y=y(x)$, in which case $x,y_k$ ($k\geq0$) will be coordinates
in the chart $J^\infty(\R,\R)\subset J^\infty(\R P^2,1)$.
Differential invariants in this case were computed by Halphen \cite{Ha}, see also \cite[Table 5]{O} and \cite{KL1,KoL}. 
The field $\mathcal{I}$ of (local) differential invariants is generated by
 \begin{equation}\label{cubic}
I =\frac{R_7}{R_5^{8/3}},\quad \nabla=\frac{R_2}{\sqrt[3]{R_5}} D_x,
 \end{equation}
where $R_2, R_5, R_7$ are relative invariants of orders 2, 5, 7 
(see \cite{KS0} for discussion of their weights) given by 
\begin{align*}
  R_2 &= y_2,  \\
  R_5 &= 9y_2^2y_5-45y_2 y_3 y_4+40 y_3^3, \\
  R_7 &= 18 y_2^4 (9 y_2^2 y_5 - 45 y_2 y_3 y_4 + 40 y_3^3) y_7 - 189 y_2^6 y_6^2
   + 126 y_2^4 (9 y_2 y_3 y_5 + 15 y_2 y_4^2 - 25 y_3^2 y_4) y_6\\ & -  189 y_2^4 (15 y_2 y_4
   + 4 y_3^2) y_5^2 + 210 y_2^2 y_3 (63 y_2^2 y_4^2 - 60 y_2 y_3^2 y_4+32 y_3^4) y_5 -  4725 y_2^4 y_4^4 \\
  & - 7875 y_2^3 y_3^2 y_4^3
   + 31500 y_2^2 y_3^4 y_4^2 - 33600 y_2 y_3^6 y_4 + 11200 y_3^8 .
\end{align*}
These invariants can be obtained by the method of moving frame as follows. 
The action is transitive on $J^1$, so we can normalize the point $p_1=\{x=0,y_0=0,y_1=0\}$. 
The stabilizer of $p_1$ is the Borel subgroup $B\subset G$ acting as follows:
 $$
(x,y)\mapsto \Bigl(\frac{ax+by}{1+\alpha x+\beta y} ,\frac{cy}{1+\alpha x+\beta y}\Bigr).
 $$

The action in $\pi_{2,1}^{-1}(p_1)$ is transitive in the complement to $R_2\neq0$. 
Thus we can normalize $p_2=\pi_{2,1}^{-1}(p_1)\cap\{y_2=1\}$, and the stabilizer $\op{St}(p_2)$ is
given by $c=a^2$. 

Next, we normalize $p_3=\pi_{3,2}^{-1}(p_2)\cap\{y_3=0\}$; the stabilizer $\op{St}(p_3)$ is
given by $b=\alpha a$. Afterwards, we normalize $p_4=\pi_{4,3}^{-1}(p_3)\cap\{y_4=0\}$, the stabilizer $\op{St}(p_4)$ is
given by $\beta=\tfrac12\alpha^2$. 

The action in $\pi_{5,4}^{-1}(p_4)$ is transitive in the complement to $R_5\equiv 9y_5\neq0$. 
Thus we can normalize $p_5=\pi_{5,4}^{-1}(p_4)\cap\{y_5=1\}$, and the stabilizer $\op{St}(p_5)$ is
given by $a=\sqrt[3]{1}$. Note that the stabilization condition $\bigl(a^{-3}y_5\bigr)|_{y_5=1}=1$ yields
three components of the stabilizer, which contribute to the deck group $\Gamma$.

Finally, the action is transitive in $\pi_{6,5}^{-1}(p_5)$ and normalizing $p_6=\pi_{6,5}^{-1}(p_5)\cap\{y_6=0\}$
uniquely fixes the last group parameter $\alpha$, which reduces the stabilizer to $\{e\}$. 

Thus we get $\Sigma=\pi_{\infty,6}^{-1}(p_6)$ and \framebox{$\Gamma=\Z_3$}. 
This explains cubic roots in \eqref{cubic}. The algebra $\mathcal{A}$ of global 
differential invariants is generated by
 \begin{equation}\label{cubic+}
I' =I^3=\frac{R_7^3}{R_5^8},\quad \nabla'=I\nabla=\frac{R_2R_7}{R_5^3} D_x
 \end{equation}
and we have
 $$
[\mathcal{I}:\mathcal{A}]=3. 
 $$
 
 \begin{remark}\label{rk3}
The deck group $\Gamma=\Z_3$ is Galois, when we work over $\C$. 
Over $\R$ the extension is not Galois, as the cubic equation has only one real root, and the normal extension
has to adjoin the cubic root of unity.
 \end{remark}

It turns out that the same extension group applies to higher-dimensional version. The route slightly changes for $n>2$.
Namely, the moving frame normalization proceeds as follows. Let $(y^1=x,y^2,\dots,y^n)$ be 
the standard projective coordinates in a chart in $\mathbb{P}^n$,
in which (generic) curves take the form $y^j=y^j(x)$. The jet-coordinates will be $x,y^j_k$ ($1<j$, $0\leq k$). 
The action of $[a_{ij}]\in G$ is given by
 $$
(y^1,y^2,\dots,y^n)\mapsto
\left(\frac{a_{10}+a_{11}y^1+a_{12}y^2+\dots+a_{1n}y^n}{a_{00}+a_{01}y^1+a_{02}y^2+\dots+a_{0n}y^n},\dots,
\frac{a_{n0}+a_{n1}y^1+a_{n2}y^2+\dots+a_{nn}y^n}{a_{00}+a_{01}y^1+a_{02}y^2+\dots+a_{0n}y^n}\right).
 $$
The prolongation of this action is transitive on $J^1$ and we can normalize 1-jet
to $p_1=\{x=0,y^j_0=0,y^j_1=0:1<j\leq n\}$. The stabilizer of this point is
the parabolic subgroup $P_{12}$ given by the upper $1+1+(n-1)$ block-triangular subgroup 
(or rather its isomorphic image in $PGL$):
 $$
\begin{pmatrix}1 & * & *\\ 0 & * & *\\ 0 & 0 & *\end{pmatrix}\subset GL(n+1)
 $$

The prolonged action on $J^n$ is transitive on $n$-jets of nondegenerate curves, which we can normalize  
to $p_n=\pi_{n,1}^{-1}(p_1)\cap\{y^j_k=\delta^j_k:1<j,k\leq n\}$.
The stabilizer of this point is the strictly upper block-triangular subgroup with $d^a_{n-1}$ a (nonstrictly)
upper-triangular matrix with the diagonal part\footnote{Warning: $i$ is the degree of $a$ in $a^i$, while $j$ in $y^j$ is 
the coordinate index.} $\op{diag}(a^2,\dots,a^n)$ and overdiagonal entries uniquely determined
by $a,\alpha$ and other $*$ in the big matrix:
 \begin{equation}\label{prjgen}
\left\{\begin{pmatrix}1 & \alpha & *\\ 0 & a & *\\ 0 & 0 & d^a_{n-1}\end{pmatrix}:a\in\R_\times,
\alpha\in\R\right\}\subset GL(n+1)
 \end{equation}

The prolonged action on $J^{n+2}$ is transitive above $p_n$, so we can normalize $(n+2)$-jet
to the point $p_{n+2}=\pi_{n+2,n}^{-1}(p_n)\cap\{y^j_k=0:k=n+1,n+2\}$.
Stabilization of $p_{n+2}$ fixes uniquely all group parameters, except for $a,\alpha$; thus, 
the stabilizer is given by \eqref{prjgen} with all un-specified entries being polynomials in $a,\alpha$. 

The remaining 2-dimensional solvable Lie group is generated by the nilpotent transformation $\exp(tv)$, $t\in\R$,
for $v=x\sum_1^ny^i\p_{y^i}-\sum_1^{n-1}q_iy^{i+1}\p_{y^i}$
(for some positive rationals $q_i$ the precise values of which are not essential) and
the scalings $(y^1,y^2,\dots,y^n)\mapsto(ay^1,a^2y^2,\dots,a^ny^n)$, $a\in\R_\times$.

Its action on $\pi_{n+3,n+2}^{-1}(p_{n+2})\simeq\R^{n-1}(y^j_{n+3}:1<j\leq n)$ 
is generated by the nilpotent $\exp(tv_{n+2})$, for $v_{n+3}=-\sum_{i=2}^{n-1}q_iy_{n+3}^{i+1}\p_{y_{n+3}^i}$
(with the same $q_i$) and the scaling $y^j_{n+3}\mapsto a^{j-n-3}y^j_{n+3}$ ($1<j\leq n$).
It follows that this action has $(n-3)$ absolute rational invariants and 1 relative invariant $y_{n+3}^n$ of weight $-3$.

A generic point in $\pi_{n+3,n+2}^{-1}(p_{n+2})$ can be normalized to the submanifold $\Sigma$ given by equations 
of $p_{n+2}$ and $\{y^{n-1}_{n+3}=0,y^n_{n+3}=1\}$. 
The stabilizer of this $\Sigma$ is given by the deck group \framebox{$\Gamma=\Z_3$} 
over $\C$ and by \framebox{$\Gamma_{\!{\sf R}}=1$} over $\R$. Therefore $\mathcal{J}=\mathcal{A}$ over $\C$, while
$\mathcal{J}=\mathcal{I}$ over $\R$, but in any case
 $$
[\mathcal{I}:\mathcal{A}]=3. 
 $$
The only difference of the case $n>2$ with the previous case is that for $n=2$ the determination of parameters 
$a,\alpha$ happens on two last jet-steps, while for $n>2$ it happens on the very last jet-level.

\section{An infinite-dimensional case: Lie problem for ODEs}\label{S7}

Consider the class of ordinary differential equations of the form 
 \begin{equation}\label{LiEq}
y''=z(x,y).
 \end{equation} 
This class of ODEs includes all Painlev\'e transcendants \cite{BB} and it was first investigated by S.\,Lie \cite{Lie2}. 
He noted that the (infinite-dimensional) pseudogroup $G$ of point transformations preserving this class is 
 \begin{equation} \label{Trans1}
x\mapsto X=a(x),\quad y\mapsto Y=c\sqrt{a'(x)}y+b(x)
 \end{equation}
for some functions $a(x),b(x)$ with $a'(x)\neq0$ and constant $c$, inducing the following transformation \cite{BB,B}
 \begin{equation} \label{Trans2}
z\mapsto Z=\frac{c}{(a')^{3/2}}z
+\frac{c}{4(a')^{7/2}}\Bigl(2a'a'''-3(a'')^2\Bigr)y
+\frac{a'b''-a''b'}{(a')^3}.
 \end{equation}

Note that $J^0=\R^3(x,y,z)$ can be considered as the space of dependent and independent variables for the problem
(in which equations \eqref{LiEq} are surfaces) and the pseudogroup $G$ lifts to $J^\infty$. 
The algebra of invariants for this action was shown in \cite{B} to have the following generators: the differential invariants
 $$
I_1 = \frac{z_{02}z_{04}}{z_{03}^2},\quad
I_2 = \frac{z_{03}}{z_{02}^3}\left((z_{22}+5z_{01}z_{02}+z_{00}z_{03}) + 
\frac{12z_{03}z_{12}z_{13} -5z_{02}z_{13}^2- 6z_{12}^2z_{04}}{5z_{02}z_{04} - 6z_{03}^2}\right),
 $$
where $z_{kl}=\p_x^k\p_y^lz$, and the invariant derivations
 $$
\nabla_1=\frac{z_{02}}{z_{03}}D_y,\quad
\nabla_2=\frac{\sqrt{z_{03}}}{z_{02}}\left(D_x-\frac{5z_{02}z_{13}-6z_{03}z_{12}}{5z_{02}z_{04}-6z_{03}^2}D_y\right).
 $$
In \cite{KS} the invariants were chosen global/rational: instead of $I_2$ there was considered
 $$
I_2' = \frac{(z_{00}z_{03} + 5z_{01}z_{02} + z_{22})(5z_{02}z_{04} - 6z_{03}^2)}{u_{02}^3z_{03}} 
- \frac{5z_{13}^2}{z_{02}^2z_{03}} - \frac{6z_{12}^2z_{0,4}}{z_{02}^3z_{03}} + \frac{12z_{12}z_{13}}{z_{02}^3},
 $$
and instead of $\nabla_2$ there was considered 
 $$
\nabla_2'= z_{03}\Delta(I_1)\Delta,\quad\text{ where }\quad
\Delta=\frac{5z_{02}z_{04} - 6z_{03}^2}{z_{02}z_{03}^2}D_x - \frac{5z_{02}z_{13} - 6z_{03}z_{12}}{z_{02}z_{03}^2}D_y.
 $$
While the first change is inessential due to syzygy $I_2'=I_2(5I_1-6)$, the second resolves the nonrationality
(radical) in the derivation $\nabla_2$. The relation between the two is noninvertible:
 $$
\nabla_2'= (5I_1-6)^2\nabla_2(I_1)\nabla_2. 
 $$

 \begin{remark}\label{rk4}
At this point let us explain why the global Lie-Tresse theorem \cite{KL2} is applicable.
Its essential assumption is algebraicity of the action of $G$, meaning that for every point $w\in J^0$ and every natural $k$
the action of the stabilizer $G^k_w$ prolonged to $k$-jets is algebraic on $\pi_{k,0}^{-1}(w)\subset J^k$.
This may seem to fail as formulae \eqref{Trans1}-\eqref{Trans2} contain radicals. Yet the action is algebraic when expressed
in the parameter $a_1^{1/2}$ and other jet-coordinates $a_i$, $i>1$ and $b_j$, $j>0$.
Thus by the main result of \cite{KL2} the orbits of $G$ can be separated by rational differential invariants
and the algebra is finitely generated in the differential sense.
 \end{remark}

Now let us describe the moving frame. The action is transitive on the base and even on $J^1$, so we can normalize the point
to $p_1=\{x=0,y=0,z_{00}=0,z_{10}=0,z_{01}=0\}$, which fixes group parameters $a_0,b_0,b_2,a_3,b_3$
(where $a_k=\p_x^ka$ and similarly for $b$).

The action is almost transitive on $\pi_{2,1}^{-1}(p_1)$ with relative invariant $z_{02}$. 
The point $p_2=p_1\cap\{z_{20}=0,z_{11}=0,z_{02}=1\}$ belongs to an open orbit,
and its stabilizer corresponds to fixed parameters $a_4,b_4,c$.

Next, the action is almost transitive on $\pi_{3,2}^{-1}(p_2)$ with relative invariant $z_{03}$. 
The point $p_3=p_2\cap\{z_{30}=0,z_{21}=0,z_{12}=0,z_{03}=1\}$ belongs to an open orbit,
and its stabilizer corresponds to fixed parameters $a_1,a_2,a_5,b_5$, however the pullback of $z_{03}$
is $a_1^2z_{03}$, thus the stabilizer of $p_3$ has two components that contribute to $\Gamma$
(this is the source of the square root in $\nabla_2$).

Furthermore, the action has orbits of codimension 2 in $\pi_{4,3}^{-1}(p_3)$ with absolute invariants 
$z_{22}-\frac{z_{13}^2}{z_{04}-6/5}$ and $z_{04}$. 
The points $p_4\in \pi_{4,3}^{-1}(p_3)\cap\{z_{40}=0,z_{31}=0,z_{13}=0\}$ belong to an open orbit,
and the stabilizer of each such point corresponds to fixed parameters $b_1,a_6,b_6$. 

After this step the situation stabilizes: for $k>4$ we get $k-1$ pure order $k$ invariants, linear in $k$-jets, 
with normalization governed by group parameters $a_{k+2},b_{k+2}$. Thus \framebox{$\Gamma=\Z_2$} and
 $$
[\mathcal{I}:\mathcal{A}]=2.
 $$
This explains the square root in local invariants of \cite{B} and the emergence of global invariants of \cite{KS}.

\section{Outlook and open problems}\label{S8}

Let us revisit Proposition \ref{KatPV} in the case of adjoint action of a semisimple complex Lie group $G\subset\op{Aut}(\g)$.
The generic stabilizer $H$ is the maximal torus, $\Sigma=\g^H=\mathfrak{h}$ is the Cartan subalgebra,
$\Gamma=N(H)=W$ is the Weyl group, and we get for rational invariants:
 \begin{equation}\label{ghW1}
\C(\g)^G=\C(\h)^W.
 \end{equation}
In this paper we generalized this for the field differential invariants as \eqref{AIG}.

A stronger result than \eqref{ghW1} is the Chevalley restriction theorem \cite{Ch2} for the ring of polynomial $\g$-invariants 
of an algebraic semisimple complex Lie algebra $\g$, aka Casimirs:
 \begin{equation}\label{ghW2}
Z(\g)=U(\g)^\g=U(\h)^W.
 \end{equation}
By the Chevalley–Shephard–Todd theorem, the latter algebra $Z(\g)$ is polynomial, so
alternatively, this formula can be written as 
 $$
\C[\g]^\g=\C[\h]^W.
 $$ 

According to \cite{KL2}, one may change $\mathcal{A}$ to the algebra $\mathfrak{A}^\ell$ consisting 
of invariants that are rational by jets of order $<\ell$ and polynomial by jets of order $\ge\ell$, 
without the finite generation and orbit separation properties.
We can therefore obtain an analog of \eqref{AIG} for rational-polynomial differential invariants.

\smallskip

The polynomiality is restored if we pass to scalar relative differential invariants. Their algebra 
$\mathcal{R}=\oplus_{w\in\mathcal{W}}\mathcal{R}_w$ is graded by a weight lattice \cite{KS}.
Almost every orbit intersects the quasi-section $\Sigma$, and we can achieve that singular hypersurfaces given by
relative invariants also intersect it; thus we can read off lowest-codimension singularities of the orbit foliation
and obtain that $\mathcal{R}$ can be identified with relative invariants of the ring of polynomials on $\Sigma$
(polynomiality is understood only in jet-coordinates, see \cite{KS0}) with respect to the deck group $\Gamma$ action.
This should be easier to compute, in particular the weight lattice $\mathcal{W}$ is expected to relate to 
characters of $\Gamma$. It can simplify computations of differential invariants, 
since every rational differential invariant is a ratio of two scalar relative invariants of the same weight \cite{KS}, 
and thus would extend algorithms of \cite{COP} to global differential invariants.

Higher-codimension singularities may not meet quasi-section $\Sigma$ obtained with the construction of Section \ref{S3}
because it is based on the union of top-dimension components of a fixed point set. Possibly allowing mixed subvarieties 
for relative sections resolves this problem, which should be also explored elsewhere.

Finally we mention that the construction of the $\Gamma$-section realizes the rational quotient $\E/G$.
It is independent of $x$ and $\Sigma$. Indeed, the former follows from the transitivity of the action of $G$
on the base of $\E$. The latter follows from the property of the quasi-section that almost every orbit intersects both 
$\Sigma$ and $\Sigma'$. Thus, if $\Sigma'$ is another quasi-section, then mapping an intersection point of an orbit 
with $\Sigma$ to an intersection point with $\Sigma'$ gives a birational isomorphism between $\Sigma/\Gamma$ and 
$\Sigma'/\Gamma'$. One should elaborate this into a constructive algorithm, based on algebraic approaches of \cite{HK1,Kem}.

\medskip

{\bf Acknowledgment.} 
The research was supported by the TFS project Lie-St\o rmer Center, the UiT Aurora project MASCOT and 
the RCN project MattNorVeg. The author is grateful for hospitality of IMPAN during the Simons Programme 
on Twistor Theory and its Applications.



\begin{thebibliography}{50}
 
\bibitem{AO}
S.\ Adams, P.\ Olver, {\it Prolonged analytic connected group actions are generically free},
Transf.\ Groups {\bf 23}, 893--913 (2018). 

\bibitem{BB} 
M.\,V.\ Babich, L.\,A.\ Bordag, {\it Projective Differential Geometrical Structure of the Painlevé Equations}, 
J.\ Differential Equations {\bf 157}, no.2, 452--485 (1999).
 
\bibitem{B}
P.\,V.\ Bibikov, {\it On Lie’s problem and differential invariants of ODEs $y''=F(x,y)$},
Funct.\ Analysis Appl.\ {\bf 51}, no. 4, 255--262 (2017).


\bibitem{CSW}
G.\ Cairns, R.\ Sharpe, L.\ Webb, {\it Conformal Invariants for Curves and Surfaces in Three Dimensional Space Forms},
Rocky Mountain J.\ Math.\ {\bf 24} (3), 933--959 (	1994).

\bibitem{C0}
E.\ Cartan, {\it Les systèmes de Pfaff, à cinq variables et les équations aux dérivées partielles du second ordre}, 
Ann. Sci. École Norm. Sup. (3) {\bf 27}, 109--192 (1910).

\bibitem{C1}
E.\ Cartan, {\it La m\'ethode du rep\`ere mobile, la th\'eorie des groupes continus, et les espaces g\'en\'eralis\'es}, 
Actualit\'es Sci.\ Ind.\ {\bf 194}, Hermann, Paris (1935). 


\bibitem{C3}
E.\ Cartan, {\it Les problèmes d'équivalence\/}, OEuvres, Partie II, vol. 2, pp. 1311--1334.

\bibitem{COP}
J.\ Cheh, P.\ Olver, J.\ Pohjanpelto, {\it Algorithms for differential invariants of symmetry groups of differential equations}, 
Found.\ Comput.\ Math.\ {\bf 8}, 501--532 (2008).


\bibitem{Ch2}
C.\ Chevalley, {\it Invariants of finite groups generated by reflections}, Amer.\ J.\ Math.\ {\bf 77}, no.4, 778--782 (1955).

\bibitem{D}
B.\ Doubrov, {\it Projective reparametrization of homogeneous curves}, Archivum Mathematicum {\bf 41}, 129--133 (2005).

\bibitem{DZ}
B.\ Doubrov, I.\ Zelenko, {\it Geometry of curves in generalized flag varieties}, Transformation Groups
{\bf 18}, no.\ 2, 361--383 (2013). 

\bibitem{FO1}
M.\ Fels, P.\ Olver, {\it Moving coframes. I. A practical algorithm}, Acta Appl.\ Math.\ {\bf 51}, no.\ 2, 161--213 (1998).

\bibitem{FO2}
M.\ Fels, P.\ Olver, {\it Moving coframes. II. Regularization and theoretical foundations}, Acta Appl.\ Math.\ {\bf 55}, no. 2, 127--208 (1999).

\bibitem{Gar}
R.\ Gardner, {\it The Method of Equivalence and Its Applications}, SIAM, Philadelphia (1989).

\bibitem{Gre}
M.\,L.\ Green, {\it The moving frame, differential invariants and rigidity theorems for curves in homogeneous spaces}, 
Duke Math.\ J.\ {\bf 45}, no.\ 4, 735--779 (1978).

\bibitem{Gri}
Ph.\ Griffiths, {\it On Cartan's method of Lie groups and moving frames as applied to uniqueness and existence questions in differential geometry}, Duke Math.\ J.\ {\bf 41}, 775--814 (1974).

\bibitem{Ha}
G.-H.~Halphen, {\it Sur les invariants différentiels}, Gauthier-Villars (1878).


\bibitem{HK1}
E.\ Hubert, I.\,A.\ Kogan, {\it Rational invariants of a group action. Construction and rewriting}, 
J.\ Symbolic Comput.\ {\bf 42}, 203--217 (2007).

\bibitem{HK2}
E.\ Hubert, I.\,A.\ Kogan, {\it Smooth and algebraic invariants of a group action. Local and global constructions},
Found.\ Comput.\ Math.\ {\bf 7}, 455--493 (2007).

\bibitem{Kat}
P.\,I.\ Katsylo, {\it Rationality of orbit spaces of irreducible representations of $SL_2$}, Izv.\ Akad.\ Nauk SSSR, 
Ser.\ Mat.\ {\bf 47}, no. 1, 26--36 (1983); English transl.: Math.\ USSR, Izv.\ {\bf 22}, 23--32 A984.

\bibitem{Kem}
G.\ Kemper, {\it The computation of invariant fields and a constructive proof of a theorem by Rosenlicht}, 
Transf.\ Groups {\bf 12}, 657--670 (2007).

\bibitem{KoL}
N.\ Konovenko, V.\ Lychagin, {\it On projective classification of plane curves}, Global and Stochastic Analysis {\bf 1} (2),
241--264 (2011).

\bibitem{KL1}
B. Kruglikov, V. Lychagin, {\it Geometry of Differential equations\/},
Handbook of Global Analysis, Ed. D.Krupka, D.Saunders, Elsevier, (2008) 725--772.

\bibitem{KL2}
B. Kruglikov, V. Lychagin, {\it Global Lie-Tresse theorem}, Selecta Math.
{\bf 22}, 1357--1411 (2016).

\bibitem{KS0}
B. Kruglikov, E.\ Schneider, {\it Invariant divisors and equivariant line bundles},
Forum of Math.\ Sigma {\bf 13}:e68, 1--36 (2025).

\bibitem{KS}
B. Kruglikov, E.\ Schneider, {\it Scalar relative differential invariants}, arXiv:2604.15473 (2026).

\bibitem{KS+}
B. Kruglikov, E.\ Schneider, {\it ODEs whose Symmetry Groups are not Fiber-Preserving}, Journal of Lie theory 
{\bf 33} (4), 1045--1086 (2023).

\bibitem{Ku}
A. Kumpera, {\it Invariants differentiels d'un pseudogroupe de Lie. I-II.\/}
J. Differential Geometry {\bf 10}, no. 2, 289--345; no. 3, 347--416 (1975).



\bibitem{Lie1}
S. Lie, {\sl Theorie der Transformationsgruppen} (Zweiter
Abschnitt, unter Mitwirkung von Prof.Dr.Friederich Engel), Teubner, Leipzig (1890).

\bibitem{Lie2}
S. Lie, {\it Ueber Differentialinvarianten}, Math. Ann. {\bf 24}, no. 4, 537--578 (1884).

\bibitem{MMR}
M.\ Magliaro, L.\ Mari, M.\ Rigoli, {\it On the geometry of curves and conformal geodesics
in the Möbius space}, Ann.\ Glob.\ Anal.\ Geom.\ {\bf 40}, 133--165 (2011).

\bibitem{MB}
G.\ Mari Beffa, {\it Relative and Absolute Differential Invariants for Conformal Curves}, 
Journ.\ Lie Theory {\bf 13}, 213--245 (2003).

\bibitem{O}
P. Olver, {\it Equivalence, Invariants, and Symmetry}, Cambridge University Press (1995).

\bibitem{OP}
P. Olver, J. Pohjanpelto, {\it Differential invariant algebras of Lie pseudo-groups\/},
Adv. Math. {\bf 222}, no. 5, 1746--1792 (2009).

\bibitem{OPV}
P. Olver, J. Pohjanpelto, F. Valiquette, {\it On the structure of Lie pseudo-groups\/},
SIGMA {\bf 5}, 077 (2009).

\bibitem{Ov}
L.\,V. Ovsiannikov, {\it Group analysis of differential
equations\/}, Russian: Nauka, Moscow (1978); Engl. transl.:
Academic Press, New York (1982).

\bibitem{PV}
V.\,L. Popov, E.\,B. Vinberg, {\it Invariant theory\/}, in: {\sl Algebraic geometry. IV\/},
Enciclopaedia of Mathematical Sciences, {\bf 55} (translation from Russian edited by
A.\,N. Parshin, I.\,R. Shafarevich), Springer-Verlag, Berlin (1994).

\bibitem{R}
M. Rosenlicht, {\it Some basic theorems on algebraic groups}, Am. J. Math. {\bf 78}, 401--443 (1956).


\bibitem{T}
T.Y. Thomas, {\it The Differential Invariants of Generalized Spaces},
Cambridge University Press, Cambridge (1934).

\bibitem{Tr1}
A. Tresse, {\it Sur les invariants differentiels des groupes continus de transformations},
Acta Math. {\bf 18}, 1--88 (1894).


\bibitem{V}
A.\,M.\ Vinogradov, {\it Local symmetries and conservation laws}, Acta Appl.\ Math. {\bf 2} (1), 21--78 (1984).

 \end{thebibliography}
\end{document}